\documentclass[]{article}
\usepackage{amsmath,amssymb,amsthm,mathrsfs,graphicx,subfigure,float,
caption,epstopdf,tabularx,bbm,color,bm,amsfonts,epic,booktabs,multirow}
\usepackage[colorlinks]{hyperref}

\usepackage{caption}
\allowdisplaybreaks[4]
\usepackage[toc,page]{appendix}
\usepackage{geometry}
\makeatletter

\newcommand{\Rmnum}[1]{\expandafter\@slowromancap\romannumeral#1@}
\makeatother

\numberwithin{equation}{section}
\newtheorem{theorem}{Theorem}[section]
\newtheorem{lemma}{Lemma}[section]

\graphicspath{{Figs/}}

\begin{document}
\title{\Large Note on a problem of Nathanson related to the $\varphi$-Sidon set}
\date{}
\author{\large Csaba S\'andor,$^{1}$\footnote{Email: csandor@math.bme.hu. This author was supported by the OTKA Grant No. K129335.}~
Quan-Hui Yang$^{2}$\footnote{Email:~yangquanhui01@163.com.}~and~Jun-Yu Zhou$^{2}$}
\date{} \maketitle
 \vskip -3cm
\begin{center}

\vskip -1cm { \small 1. Institute of Mathematics, Budapest
University of Technology and Economics, and MTA-BME Lend\"{u}let Arithmetic Combinatorics Research Group, ELKH, H-1529 B.O. Box, Hungary}
 \end{center}

 \begin{center}
{ \small 2. School of Mathematics and Statistics, Nanjing University of Information \\
Science and Technology, Nanjing 210044, China}
 \end{center}

\begin{abstract} Let $\varphi (x_{1}, \ldots, x_{h})=c_{1} x_{1}+\cdots+c_{h} x_{h}$ be a linear form with coefficients in a field $\mathbf{F}$, and let $V$ be a vector space over $\mathbf{F}$. A nonempty subset $A$ of $V$ is a $\varphi$-Sidon set if $\varphi\left(a_{1}, \ldots, a_{h}\right)=\varphi\left(a_{1}^{\prime}, \ldots, a_{h}^{\prime}\right)$ implies $\left(a_{1}, \ldots, a_{h}\right)=$ $\left(a_{1}^{\prime}, \ldots, a_{h}^{\prime}\right)$ for all $h$-tuples $\left(a_{1}, \ldots, a_{h}\right) \in A^{h}$ and $\left(a_{1}^{\prime}, \ldots, a_{h}^{\prime}\right) \in A^{h}$.
We call $A$ a  polynomial perturbation of $B$ if for some $r>0$ and positive integer $k_0$,
$|a_k-b_k|< k^r$ holds for all integers $k \geq k_0$. In this paper, for a given set $B$, we prove
that there exists a $\varphi$-Sidon set $A$ of integers that is a polynomial perturbation of $B$.
This gives an affirmative answer to a recent problem of Nathanson. Some other results are also
proved.

\end{abstract}

\section{Introduction}
For a field $\mathbf{F}$ and a positive integer $h$, we define linear forms
\begin{eqnarray}\label{eq11}
\varphi(x_1,\ldots,x_h)=c_1 x_1+\cdots+c_h x_h,
\end{eqnarray}
where $c_i\in \mathbf{F}$ for all $i\in\{1,\ldots,h\}$. Let $V$ be a vector space over the field $\mathbf{F}$. For every nonempty set $A\subseteq V$, let
$$
A^h=\left\{ (a_1,a_2,\ldots,a_h):a_i\in A~\text{for~all}~i\in \{1,2,\ldots,h\} \right\}
$$
be the set of all $h$-tuples of elements of $A$. For $c\in \mathbf{F}$, the $c$-$dilate$ of $A$ is defined as
$$
c\ast A=\{ca: a\in A \}.
$$
The $\varphi$-$image$ of $A$ is the set
\begin{eqnarray*}
 \varphi(A)&=&\left\{ \varphi(a_1,a_2,\ldots,a_h):(a_1,a_2,\ldots,a_h)\in A^h \right\}\\
 &=&\{c_1a_1+\cdots+c_ha_h:~(a_1,\ldots,a_h)\in A^h\}\\
 &=& c_1 \ast A+ \cdots +c_h \ast A .
\end{eqnarray*}

We call a nonempty subset $A$ of $V$ a $Sidon$ set for the linear form $\varphi$ (or a $\varphi$-$Sidon$ set)
if $ \varphi(a_1,a_2,\ldots,a_h)$ with $(a_1,a_2,\ldots,a_h)\in A^h$ are all distinct. That is, for all $h$-tuples $(a_1,a_2,\ldots,a_h)\in A^h$ and $(a'_1,a'_2,\ldots,a'_h)\in A^h$, if
$\varphi (a_1,a_2,\ldots,a_h) =\varphi(a'_1,a'_2,\ldots,a'_h),$
then $(a_1,a_2,\ldots,a_h)= (a'_1,a'_2,\ldots,a'_h)$.

For every nonempty subset $I$ of $\{1,\ldots,h\}$, define the subset sum
\begin{eqnarray}\label{eq12}
s_I =\sum_{i\in I}c_i .
\end{eqnarray}
Let $s_\emptyset =0$. Suppose there exist disjoint subsets $I_1$ and $I_2$ of $\{1,\ldots,h\}$ with $I_1$ and $I_2$ not both empty such that
\begin{eqnarray}\label{eq13}
s_{I_{1}} =\sum_{i\in I_1}c_i =\sum_{i\in I_2}c_i=s_{I_{2}}.
\end{eqnarray}
Then $A$ is not a sidon set (See \cite{Nathanson1}).
We say that the linear form (\ref{eq11}) has property $N$ if there do not exist disjoint nonempty
subsets $I_1$ and $I_2$ of $\{1,\ldots,h\}$ that satisfy (\ref{eq13}).

Let $J\subseteq \{1, \ldots, h\}$, we define the linear form in $\operatorname{card}(J)$ variables
$$
\varphi_{J}=\sum_{j \in J} c_{j} x_{j}.
$$
By definition, $\varphi_{\emptyset}=0$ and $\varphi_{J}=\varphi$ if $J=\{1, \ldots, h\} .$ The linear form $\varphi_{J}$ is called a contraction of the linear form $\varphi$.
For every nonempty subset $A$ of $V$, let
$$
\varphi_{J}(A)=\left\{\sum_{j \in J} c_{j} a_{j}: a_{j} \in A \text { for all } j \in J\right\}.
$$
If $A$ is a $\varphi$-Sidon set, then $A$ is a $\varphi_{J}$-Sidon set for every nonempty subset $J$ of $\{1, \ldots, h\}$.

For every subset $X$ of $V$ and vector $v \in V$, the {\em translate} of $X$ by $v$ is the set
$$
X+v=\{x+v: x \in X\}.
$$
For every subset of $J$ of $\{1, \ldots, h\}$, let $J^{c}=\{1, \ldots, h\} \backslash J$ be the complement of $J$ in $\{1, \ldots, h\}$. For every subset $A$ of $V$ and $b \in V \backslash A$, we define
$$
\Phi_{J}(A, b)=\varphi_{J}(A)+\bigg(\sum_{j \in J^{c}} c_{j}\bigg) b=\varphi_{J}(A)+s\left(J^{c}\right) b
$$
be the translate of the set $\varphi_{J}(A)$ by the subset sum $s\left(J^{c}\right) b$. We have $\Phi_{\emptyset}(A, b)=\left(\sum_{j=1}^{h} c_{j}\right) b$ and $\Phi_{J}(A, b)=\varphi(A)$ if $J=\{1, \ldots, h\}$.

Let $A=\{a_k:k=1,2,3,\ldots \}$ and $B=\{b_k:k=1,2,3,\ldots \}$ be sets of integers. We define $A$ a {\em polynomial perturbation} of $B$ if for some $r>0$ and positive integer $k_0$,
$$ |a_k-b_k|< k^r$$
for all integers $k \geq k_0$. The set $A$ is a {\em bounded perturbation} of $B$ if
there exists an $m_0 > 0$ such that
$$ |a_k-b_k|< m_0$$
holds for all integers $k \geq k_0$.

Recently, Nathanson \cite{Nathanson1} posed the following problem.

\noindent{\bf Nathanson's Problem} {\em  Let $\varphi$ be a linear form with integer coefficients that satisfies condition $N$. Let $B$ be a set of integers. Does there exist a $\varphi$-Sidon set of integers that is a polynomial perturbation of $B$? Does there exist a $\varphi$-Sidon set of integers that is a bounded perturbation of $B$?}

For other related results about Sidon set,
one can refer to [1]-[4],[6],[7]. In this paper, we give an affirmative answer to the previous problem and give
some partial results to the last problem. 

\begin{theorem}\label{thm1}
Let $\varphi=\Sigma_{i=1}^h c_i x_i$ be a linear form with integer coefficients that satisfy condition $N$ and $B=\{b_k:k=1,2,3,\ldots \}$ be a set of integers. Then there exists an infinite $\varphi$-Sidon set $A=\{a_k:k=1,2,3,\ldots \}$ of integers such that
$|a_k-b_k|< k^{4h}$ holds for all positive integers $k$.
\end{theorem}

\begin{theorem}\label{thm2}
Let $\varphi=\Sigma_{i=1}^h c_i x_i$ be a linear form with integer coefficients that satisfies condition $N$, $C=\sum^{h}_{i=1}|c_i|$, and $B=\{b_1<b_2<\cdots\}$ be
a set of integers. For any $\epsilon>0$, if $|b_t - b_s|\leq \left(  t-s+1  \right)^{h-\epsilon}$ for any positive integers $t>s$, then there does not exist a $\varphi$-Sidon set of integers that is a polynomial perturbation of $B$.
\end{theorem}

\begin{theorem}\label{thm3}
Let $\varphi=\sum_{i=1}^h c_i x_i$ be a linear form with integer coefficients that satisfies condition $N$, $C=\sum^{h}_{i=1}|c_i|$, and $B=\{b_1<b_2<\cdots\}$ be
a set of integers. If $b_1>m$ and $b_{k+1}> C b_k+(C+1)m$ for some $m\ge 0$ and
for all positive integers $k$, then there exists a $\varphi$-Sidon set
$A=\{a_1,a_2,\ldots\}$ of integers and a constant $m_0>0$ such that
$ |a_k-b_k|<m_0 $ for all positive integers $k$.
\end{theorem}

\section{Proofs}
\begin{lemma}\label{lem21}\cite[Lemma 1]{Nathanson1}
Let $\varphi=\Sigma_{i=1}^h c_i x_i$ be a linear form with coefficients in the field $\mathbf{F}$. Let $V$ be a vector space over $\mathbf{F}$.
For every subset $A$ of $V$ and $b\in V\backslash A$,
$$
\varphi \left(A\cup \{b\}\right)= \bigcup_{J\subseteq \{1,\ldots,h \}} \Phi_J (A,b).
$$
If $A\cup \{b\}$ is a $\varphi$-Sidon set, then
\begin{eqnarray}\label{eq2.0}
\left\{   \Phi_J(A,b): J\subseteq \{1,\ldots,h \}  \right\}
\end{eqnarray}
is a set of pairwise disjoint sets.

If $A$ is a $\varphi$-Sidon set and (\ref{eq2.0}) is a set of pairwise disjoint sets, then $A\cup \{b\}$ is a $\varphi$-Sidon set.

\end{lemma}

\begin{lemma}\label{lem22}
If $A$ is a $\varphi$-Sidon set, then any subset of $A$ is also a $\varphi$-Sidon set.
\end{lemma}

The Proof of Lemma \ref{lem22} is easy, we leave it to the reader.
%\begin{proof}
%Let $A_n$ is a subset of $A$, and so $A^h_n \subseteq A^h$. $A$ is a $\varphi$-Sidon set, it follows that for all $h$-tuples $(a_1,\ldots,a_h)\in A^h_n \subseteq A^n$ and $(a'_1,a'_2,\ldots,a'_h)\in A^h_n \subseteq A^n$, if
%$$
%\varphi (a_1,a_2,\ldots,a_h) =\varphi(a'_1,a'_2,\ldots,a'_h),
%$$
%then $(a_1,a_2,\ldots,a_h)= (a'_1,a'_2,\ldots,a'_h)$,  and by the definition of $\varphi$-Sidon set, we get that $A_n$ is also a $\varphi$-Sidon set.
%
%\end{proof}

\begin{proof}[Proof of Theorem \ref{thm1}]
We construct the $\varphi$-Sidon set $A=\{a_k:k=1,2,3,\ldots \}$ inductively. Let $a_1=b_1$, then $A_1=\{a_1\}$ is a $\varphi$-Sidon set and $|a_1-b_1|=0< 1^{4h}$.

Let $k\geq 1$ and let $A_k=\{a_1,a_2,\ldots,a_k\}$ be a set of $k$ distinct positive integers such that $A_k$ is a $\varphi$-Sidon set with $|a_i-b_i|< i^{4h}$ for all integers $i\leq k$. Let $b$ be a positive integer. By Lemma \ref{lem21}, the set $A_k\cup \{b\}$ is a $\varphi$-Sidon set if and only if the sets
$$\Phi_{J}(A_k,b)=\varphi_{J} (A_k)+\bigg(\sum_{j\in J^{c}} c_j\bigg)b$$
are pairwise disjoint for all $J\subseteq\{1,\ldots,h\}$. Let $J_1$ and $J_2$ be distinct subsets of $\{1,\ldots,h\}$. We have
$$ \Phi_{J_{1}}(A_k,b)\cap \Phi_{J_{2}}(A_k,b)\neq \emptyset$$
if and only if there exist integers $a_{1,j}\in A_k$ for all $j\in J_1$ and $a_{2,j}\in A_k$ for all $j\in J_2$ such that
$$ \sum_{j\in J_{1}}c_j a_{1,j}+\bigg(\sum_{j\in J_{1}^{c}} c_j\bigg)b= \sum_{j\in J_{2}}c_j a_{2,j}+\bigg(\sum_{j\in J_{2}^{c}} c_j\bigg)b.$$
Equivalently, the integer $b$ satisfies the equation
\begin{eqnarray}\label{eq2.2}
\bigg(\sum_{j\in J^{c}_{2}}c_j-\sum_{j\in J^{c}_{1}}c_j\bigg)b=\sum_{j\in J_{1}}c_j a_{1,j}-\sum_{j\in J_{2}}c_j a_{2,j}.
\end{eqnarray}
The integer
$$\begin{aligned}
 c=\sum_{j\in J^{c}_{2}}c_j-\sum_{j\in J^{c}_{1}}c_j = s(J_{2}^{c})-s(J_{1}^{c})
 =s(J_{1}\backslash (J_{1}\cap J_{2}) )-s(J_{2}\backslash (J_{1}\cap J_{2}) )
\end{aligned}$$
is nonzero because $J_{1}\backslash (J_{1}\cap J_{2})$ and $J_{2}\backslash (J_{1}\cap J_{2})$ are distinct and disjoint, and the linear form $\varphi$ satisfies condition $N$, and so there exists at most one integer $b $ that satisfies equation (\ref{eq2.2}).

Let $card(J_1)=k_1$ and $card(J_2)=k_2$. The sets $J_1$ and $J_2$ are distinct subsets of $\{1,\ldots,h\}$, and so at least one of the sets $J_1$ and $J_2$ is a proper subset of $\{1,\ldots,h\}$. It follows that
$$ k_1+k_2=card(J_1)+card(J_2)\leq 2h-1.$$
The number of integers of the form
$$\sum_{j\in J_{1}}c_j a_{1,j}-\sum_{j\in J_{2}}c_j a_{2,j} $$
with $a_{1,j}\in A_n$ and $a_{2,j}\in A_n$ is at most $n^{k_1+k_2}$. The number of ordered pairs $\left( J_1,J_2 \right)\subseteq \{1,\ldots,h\}$ with $|J_1|=k_1$ and $|J_2|=k_2$ is
$ \dbinom{h}{k_1}\dbinom{h}{k_2}.$
Thus, the number of equations of the form (\ref{eq2.2}) is less than
$$ \sum^{h}_{k_1=0}\sum^{h}_{k_2=0}\dbinom{h}{k_1}\dbinom{h}{k_2}n^{k_1+k_2}<\left(  \sum^{h}_{k_1=0}\dbinom{h}{k_1}  \right)^2 n^{2h-1}=4^h n^{2h-1},$$
and so there are less than $4^h n^{2h-1}+n$ positive integers $b$ such that $b \notin A_n$ and $A_n\cup \{b\}$ is a $\varphi$-Sidon set.

Let $a_{n+1}\in \left( b_{n+1}-4^h n^{2h-1}-n,b_{n+1}+4^h n^{2h-1}+n \right)$ satisfy that $A_{n+1}=A_{n}\cup \{a_{n+1}\}$ is a $\varphi$-Sidon set. Then
$$|a_{n+1}-b_{n+1}|<4^h n^{2h-1}+n.$$

Next, we prove that $4^h n^{2h-1}+n<(n+1)^{4h}$ for all positive integers $n$.

Case 1. $n=1$. Since $h$ is a positive integer and $2^{2h}\geq 4$, we have
$$2^{4h}=\left(2^{2h}\right)^2>2^{2h}+1=4^h+1.$$

Case 2. $n\geq 2$. It follows that
$$\begin{aligned}
(n+1)^{4h}&=(n+1)^{2h} \cdot (n+1)^{2h}
          > 2^{2h} \cdot n^{2h}= 4^h \cdot n^{2h-1} \cdot n \\
          &\geq 4^h \cdot n^{2h-1} \cdot 2
          > 4^h \cdot n^{2h-1}+ n.
\end{aligned}$$

This completes the proof of Theorem \ref{thm1}.
\end{proof}

\begin{proof}[Proof of Theorem \ref{thm2}]
Suppose that there exist a $\varphi$-Sidon set $A=\{a_k:k=1,2,\ldots\} $ and an integer $m_0$ such that
$ |a_k-b_k|<m_0$ for all positive integers $k$.
Let $A_{t-s}=\{a_s,a_{s+1},\ldots,a_t \}, t\geq s$. By Lemma \ref{lem22}, $A_{t-s}$ is also a $\varphi$-Sidon set, so we have
$$ \varphi\left(A_{t-s}\right)=\left\{ \sum_{i=1}^{h} c_i a_i :a_i\in A_{t-s} \right\} $$
and
\begin{eqnarray}\label{eq2.3}
|\varphi\left(A_{t-s}\right)|=(t-s+1)^h.
\end{eqnarray}
Let 
$$J_1=\{i:~1\le i\le h ~\text{and}~c_i>0\},\quad J_2=\{i:~1\le i\le h~\text{and}~c_i<0\}.$$
Then
$$
\varphi\left(A_{t-s}\right)_{\max}=\sum_{i\in J_1} c_i\cdot \max(A_{t-s}) +\sum_{i\in J_2} c_i \cdot \min(A_{t-s})
<\sum_{i\in J_1} c_i (b_t+m_0) +\sum_{i\in J_2} c_i (b_s-m_0),
$$
$$
\varphi\left(A_{t-s}\right)_{\min}=\sum_{i\in J_2} c_i \cdot \max(A_{t-s}) +\sum_{i\in J_1} c_i \cdot \min(A_{t-s})
>\sum_{i\in J_2} c_i (b_t+m_0) +\sum_{i\in J_1} c_i (b_s-m_0).
$$
Since $\varphi\left(A_{t-s}\right)$ is a set of integers, it follows that
\begin{eqnarray}\label{eq2.4}
\begin{aligned}
|\varphi\left(A_{t-s}\right)|&\le \sum_{i\in J_1} c_i (b_t+m_0) +\sum_{i\in J_2} c_i (b_s-m_0)- \bigg(\sum_{i\in J_2} c_i (b_t+m_0) +\sum_{i\in J_1} c_i (b_s-m_0)\bigg)-1  \\
          &=  \sum_{i=1}^{h} |c_i| (b_t+m_0)- \sum_{i=1}^{h} |c_i| (b_s-m_0) -1   \\
          &= C(b_t - b_s + 2 m_0)-1.
\end{aligned}
\end{eqnarray}
By (\ref{eq2.3}) and (\ref{eq2.4}), we have
$$
(t-s+1)^h\le C(b_t - b_s + 2 m_0)-1 \leq C\left[(t-s+1)^{h-\epsilon} + 2 m_0\right]-1.
$$
This inequality can not hold when $t-s+1$ large enough, a contradiction.

This completes the proof of Theorem \ref{thm2}.
\end{proof}

\begin{proof}[Proof of Theorem \ref{thm3}]
Take $m_0$ with $0<m_0 \leq m$ and $a_1=b_1$, then $A_1=\{a_1\}$ is a $\varphi$-Sidon set and  $|a_1-b_1|=0< m_0$.

Now we prove by induction on $k$. Suppose that $k\geq 1$ and $A_k=\{a_1,a_2,\ldots,a_k\}$ is a $\varphi$-Sidon set of integers such that $|a_i-b_i|<m_0$ for $i=1,2,\ldots,k$. 
Take a positive integer $a_{k+1}$ such that $$a_{k+1}\in \left(b_{k+1}-m_0,b_{k+1}+m_0 \right).$$
By Lemma \ref{lem21}, the set $A_k\cup \{a_{k+1}\}$ is a $\varphi$-Sidon set if and only if the sets
$$\Phi_{J}(A_k,a_{k+1})=\varphi_{J} (A_k)+\bigg(\sum_{j\in J^{c}} c_j\bigg)a_{k+1}$$
are pairwise disjoint for all $J\subseteq\{1,\ldots,h\}$.
Suppose that there exist two distinct subsets $J_1,J_2\subseteq \{1,\ldots,h\}$ such that
$$ \Phi_{J_{1}}(A_k,a_{k+1})\cap \Phi_{J_{2}}(A_k,a_{k+1})\neq \emptyset.$$
It follows that there exist integers $a_{1,j}\in A_k$ for all $j\in J_1$ and $a_{2,j}\in A_k$ for all $j\in J_2$ such that
$$ \sum_{j\in J_{1}}c_j a_{1,j}+\bigg(\sum_{j\in J_{1}^{c}} c_j\bigg)a_{k+1}= \sum_{j\in J_{2}}c_j a_{2,j}+\bigg(\sum_{j\in J_{2}^{c}} c_j\bigg)a_{k+1}.$$
Equivalently, the integer $a_{k+1}$ satisfies the equation
$$
\bigg(\sum_{j\in J^{c}_{2}}c_j-\sum_{j\in J^{c}_{1}}c_j\bigg)a_{k+1}=\sum_{j\in J_{1}}c_j a_{1,j}-\sum_{j\in J_{2}}c_j a_{2,j}.
$$
Since $J_{1}^{c}\neq J_{2}^{c}$ and the linear form $\varphi$ satisfies condition $N$,
it follows that
$\sum_{j\in J^{c}_{2}}c_j-\sum_{j\in J^{c}_{1}}c_j $
is nonzero, and so
$$
a_{k+1}=\frac{\sum_{j\in J_{1}}c_j a_{1,j}-\sum_{j\in J_{2}}c_j a_{2,j}}{\sum_{j\in J^{c}_{2}}c_j-\sum_{j\in J^{c}_{1}}c_j}\le |\sum_{j\in J_{1}}c_j a_{1,j}-\sum_{j\in J_{2}}c_j a_{2,j}|
\le \sum_{i=1}^h |c_i|\cdot \max\{a_1,a_2,\ldots,a_k\}
<C (b_k+m_0).
$$
Thus we have
$$
a_{k+1}>b_{k+1}-m> C b_k+(C+1)m-m=C (b_k+m)\geq C (b_k+m_0)> a_{k+1},
$$
which is a contradiction. Hence $ \Phi_{J_{1}}(A_k,a_{k+1})\cap \Phi_{J_{2}}(A_k,a_{k+1})=\emptyset.$

Therefore, the set $A_k\cup \{a_{k+1}\}$ is a $\varphi$-Sidon set and $|a_i-b_i|<m_0$ for all $i\leq k+1$.
\end{proof}

%\begin{theorem}\label{thm1} Let $k\ge 2$ be an integer and $A=\{a_0,a_1,\ldots,a_k\}$ be
%a set of integers satisfying
%$0=a_0<a_1<\cdots<a_k$ and $(a_1,a_2,\ldots,a_k)=1$. Then
%$$(hA)^{(t)}=C_t\cup [c,ha_k-d]\cup (ha_k-D_t)$$
%for all integers $h\ge h_t$, where
%$c$, $d$ are positive integers, $h_t=\sum_{i=2}^k (ta_i-1)-1$, and
%$C_t\subseteq [0,c-2]$ and $D_t\subseteq [0,d-2]$.
%\end{theorem}
%
%\begin{proof}[Proof]
%\end{proof}

%\section{Acknowledgement}


\begin{thebibliography}{30}



\bibitem{Cilleruelo1} J. Cilleruelo, {\it On Sidon sets and asymptotic bases,}
Proc. Lond. Math. Soc. 111 (2015), 1206-1230.

\bibitem{Cilleruelo2} J. Cilleruelo, {\it Infinite Sidon sequences,} Adv. Math. 255 (2014), 474-486.

\bibitem{fang} J.-H. Fang, {\it On generalized perfect difference sets constructed from Sidon sets,} Discrete Math. 344 (2021), Paper No. 112589.

\bibitem{Pach} P.P. Pach, {\it An improved upper bound for the size of the multiplicative 3-Sidon sets,} Int. J. Number Theory 15 (2019), 1721-1729.

\bibitem{Nathanson1} M.B. Nathanson, {\it Sidon sets for linear forms,} arXiv:2101:01034, 2021.

\bibitem{Nathanson2} M.B. Nathanson, {\it An inverse problem for finite Sidon sets,} arXiv:2104.06501, 2021.

\bibitem{Nathanson3} M.B. Nathanson, {\it The Bose-Chowla argument for Sidon sets,} arXiv:2104.12711 , 2021.


\end{thebibliography}
\end{document}